\documentclass[11pt]{amsart}

\usepackage[T1]{fontenc}
\usepackage{lmodern}
\usepackage{microtype}
\usepackage{amsmath,amssymb,amsthm,mathtools}
\usepackage{enumitem}
\usepackage{xcolor}
\usepackage{hyperref}
\usepackage{graphicx}
\usepackage[nameinlink,capitalise]{cleveref}

\hypersetup{
  colorlinks=true,
  linkcolor=blue!55!black,
  citecolor=green!45!black,
  urlcolor=blue!65!black,
  pdftitle={A stationary harmonic map with a nonstationary tangent map},
  pdfauthor={Draft for discussion},
  pdfsubject={Stationary harmonic maps, tangent maps, and defect measures},
  pdfkeywords={harmonic map, stationary map, tangent map, defect measure,
    Kahler surface, holomorphic curve}
}

\newtheorem{theorem}{Theorem}[section]
\newtheorem{proposition}[theorem]{Proposition}
\newtheorem{lemma}[theorem]{Lemma}

\theoremstyle{definition}

\theoremstyle{remark}
\newtheorem{remark}[theorem]{Remark}

\newcommand{\R}{\mathbb R}
\newcommand{\C}{\mathbb C}
\newcommand{\CP}{\mathbb {CP}}
\newcommand{\Sph}{\mathbb S}
\newcommand{\cH}{\mathcal H}
\newcommand{\cD}{\mathcal D}
\newcommand{\cS}{\mathcal S}
\newcommand{\cU}{\mathcal U}
\newcommand{\cM}{\mathcal M}
\newcommand{\ddc}{dd^c}

\newcommand{\diver}{\operatorname{div}}
\newcommand{\Bl}{\operatorname{Bl}}

\newcommand{\Id}{\operatorname{Id}}
\newcommand{\supp}{\operatorname{spt}}

\newcommand{\vol}{\operatorname{vol}}
\newcommand{\weakto}{\rightharpoonup}
\newcommand{\dd}{{\rm d}}
\newcommand{\G}{\Gamma}
\newcommand{\na}{\nabla}
\newcommand{\e}{{\varepsilon}}

\newcommand{\p}{\partial}
\newcommand{\loc}{{\rm loc}}
\newcommand{\weak}{\rightharpoonup}

\newcommand{\mres}{\mathbin{\vrule height 1.6ex depth 0pt width
0.13ex\vrule height 0.13ex depth 0pt width 1.3ex}}

\def\XXint#1#2#3{{\setbox0=\hbox{$#1{#2#3}{\int}$ }
\vcenter{\hbox{$#2#3$ }}\kern-.6\wd0}}

\newcommand{\ball}[2]{\mathbf{B}_{#1}({#2})}

\newcommand{\E}{\mathcal{E}}

\newcommand{\stwo}{{\bf S}^2}
\newcommand{\bthree}{{\bf B}^3}
\newcommand{\sing}{{\rm sing}}
\newcommand{\singu}{\sing(u)}

\newcommand{\leb}{{\mathcal{L}}}
\newcommand{\hau}{{\mathcal{H}}}

\newcommand{\revisioncolor}{\color{blue}%
  \hypersetup{linkcolor=blue,citecolor=blue,urlcolor=blue}}

\title[A nonstationary tangent map]{A stationary harmonic map with a nonstationary tangent map}

\author{Siran Li}

\address{Siran Li: School of Mathematical Sciences $\&$ CMA-Shanghai, Shanghai Jiao Tong University, No.~6 Science Buildings,
800 Dongchuan Road, Minhang District, Shanghai, China (200240)}

\email{\texttt{siran.li@sjtu.edu.cn}}

\author{Zijun Wang}

\address{Zijun Wang: Institute for Theoretical Sciences, Westlake University, 600 Dunyu Road, Xihu district, Hangzhou, China (310030)}

\email{\texttt{wangzijun79@westlake.edu.cn}} 


\date{\today}

\begin{document}

\begin{abstract}
We construct a stationary weakly harmonic map from a domain in $\R^5$ into a compact K\"ahler surface for which a tangent map is weakly harmonic but nonstationary. The target manifold is the blow-up of $\CP^1\times\CP^1$ at four points lying in one fiber and arranged at the vertices of a regular tetrahedron. The nonstationarity of the tangent map arises from the break down of the tetrahedral symmetry in the defect measure concentrated on the vertices.

\end{abstract}

\maketitle


\section{Introduction}

 Let $\Omega \subset \R^n$ be a Euclidean domain and $N$ be a compact Riemannian manifold. A mapping $u \in W^{1,2}(\Omega,N)$ is a \emph{minimizing} harmonic map if it minimizes the Dirichlet energy $\mathcal{E}_\Omega[u]:=\int_\Omega |\na u|^2 \,\dd x$ among all the $W^{1,2}$-maps with the same trace on the boundary. A mapping is \emph{weakly harmonic} if it is a weak solution of the harmonic map equation, \textit{i.e.}, the Euler--Lagrange equation with respect to compactly supported outer variations in the target variable. 
 
 It is \emph{stationary} harmonic if it is also critical under compactly supported inner variations of the domain. It is well-known that minimizing harmonic maps are stationary and stationary harmonic maps are weakly harmonic, but neither of the reverse implications holds true. We refer to the surveys or expositions by Simon~\cite{simon}, Hardt~\cite{hardt}, and Lin--Wang~\cite{lw}, among many others, for detailed discussions.
See also Lin--Rivi\`ere~\cite{LinRiviere2002} on energy
quantization and Rivi\`ere~\cite{riv} on everywhere
discontinuous weakly harmonic maps.

\subsection{Background}

The regularity theory for minimizing and stationary harmonic maps developed in the seminal works \cite{SchoenUhlenbeck1982,Evans1991,Bethuel1993} by Schoen--Uhlenbeck, Evans, Bethuel, as well as many others, has long been a central topic in geometric analysis. Many of the most important questions on harmonic maps concern stationary harmonic maps, including energy quantization and the codimension-three conjecture for the singular set.
These questions were resolved, respectively, by Naber--Valtorta~\cite{NaberValtorta2024} and, in the negative, by Jiang--Li--Yu~\cite{jiang}.

Various powerful general machineries have been developed for stationary harmonic maps. These include, \textit{e.g.}, the defect measure and concentration set analysis in Lin~\cite{Lin1999} and the quantitative rectifiability theory in Cheeger--Naber~\cite{cn} and Naber--Valtorta~\cite{nv}. Despite these developments, geometric and analytic properties of stationary harmonic maps remain elusive in the large. This is true, in particular, of the limiting behavior of sequences of weakly convergent stationary harmonic maps.

Li--Tian~\cite{LiTian1998} proved a blow-up formula which, in particular, implies that tangent maps of stationary harmonic maps on three-dimensional domains are stationary; see also the later developments in Li--Zhou--Zhu~\cite{LiZhouZhu2023}. For minimizing harmonic maps, Schoen--Uhlenbeck~\cite{SchoenUhlenbeck1982} proved via $\e$-regularity that any sequence of minimizing harmonic maps weakly convergent in $W^{1,2}(\Omega;N)$ indeed converges strongly in $W^{1,2}_{\rm loc}(\Omega;N)$; thus, by the Luckhaus lemma~\cite{luck} (\textit{cf.} \cite{hl, SchoenUhlenbeck1982} too), the limiting map is also a minimizing harmonic map. 

Li--Tian also conjectured in \cite[Section~5, p.752]{LiTian1998} that the weak limiting map of any sequence of stationary harmonic maps is also stationary. Ding--Li--Li~\cite{DingLiLi2003} disproved this conjecture by constructing a  sequence of stationary harmonic maps from $\bthree$ to $\stwo$ whose weak limit is nonstationary. Moreover, Hardt--Lin--Poon~\cite{hlp} constructed a stationary, axially symmetric harmonic map $u:\bthree\to\stwo$, smooth away from isolated singularities of degree zero. Indeed, $\singu \cap \ball{\delta_0}{0} = \{0\}$ for some $\delta_0>0$ and $\deg \left(u \mres \p\ball{\delta}{0} \right) = 0$ for all $0<\delta<\delta_0$. However, the rescaled mappings $u_\lambda (x) := u(\lambda x)$ converge weakly to a constant, while 
$$|\na u_\lambda|^2\dd\leb^3 \mres \bthree \weak 16\pi \hau^1\mres \left[(\{0\}\times \R^1)\cap{\bthree}\right]$$ as $\lambda \to 0^+$. 
This example shows that nontrivial energy concentration may persist even when the weak tangent map is constant.

To the best of our knowledge, it remains unclear if blow-up sequences of stationary harmonic maps obtained by rescaling at a singular point always converge weakly to a stationary tangent map. Recall that a tangent map of $u:\Omega\to N$ at a point $y\in\Omega$ is a map $\phi:\R^n\to N$ obtained as a weak $W^{1,2}_{\mathrm{loc}}$-limit
\[
u_{y,\rho_j}\rightharpoonup\phi,
\qquad
u_{y,\rho_j}(x):=u(y+\rho_jx),
\qquad \rho_j\downarrow0.
\]
Here $u_{y,\rho_j}$ is defined on $\rho_j^{-1}(\Omega-y)$, and local convergence is understood on each fixed ball for all sufficiently large $j$.

\subsection{Main Theorem}
In this paper, we answer the above question in the negative. By exploiting ideas from K\"ahler geometry and bubble analysis, we construct an explicit example for a stationary harmonic map whose tangent map is nonstationary. Our main result is as follows:

\begin{theorem}\label{thm:main}
There exist a compact K\"ahler surface $(X,h)$, an open neighborhood $\Omega\subset\R^5$ of the origin, and a stationary harmonic map $u\in W^{1,2}(\Omega,X)$
such that the rescalings $u_\rho(x)=u(\rho x)$ converge weakly in $W^{1,2}_{\mathrm{loc}}(\R^5,X)$ as $\rho\downarrow0$ to a tangent map $\phi$ that is weakly harmonic but nonstationary.
\end{theorem}

\begin{remark}\label{rem:higher-dimensional-extension}
In Theorem~\ref{thm:main}, the weak tangent map $\phi$ at the origin is unique, since the whole family $\{u_\rho\}$ converges weakly to $\phi$ in $W^{1,2}_{\mathrm{loc}}(\R^5,X)$.
For every integer $n>5$, extending $u$ constantly in the additional variables gives an example on a neighborhood of the origin in $\R^n$, with the same target $(X,h)$ and a unique nonstationary weak tangent map at the origin.
\end{remark}

To prove Theorem~\ref{thm:main}, we  take the target manifold $X$ to be
\[
 X=\Bl_{\{(p_k,q_0)\}_{1\le k\le4}}
\left(\CP^1\times\CP^1\right),
\]
where $q_0\in\CP^1$ is an arbitrary point, and $p_1,\dots,p_4\in\CP^1\simeq\stwo$ are the vertices of a regular tetrahedron in $\R^3$. As $X$ is a compact projective manifold, it is  K\"{a}hler.  Let $$\beta:X \longrightarrow \CP^1\times\CP^1 $$ be the associated blow-up map. The second $\CP^1$ factor of $\CP^1\times\CP^1$ parametrizes the fibers of ${\rm pr}_2\circ\beta$, which degenerate as the parameter approaches $q_0$. In local coordinates centered at $(p_k,q_0)$, the blow-up map is given by $\beta(w,t)=(w,wt)$, with $a=wt$. We construct a weakly harmonic map $u:\Omega\to X$ such that (1) a perturbation of a tetrahedrally symmetric K\"ahler form,
together with an exact balancing correction, makes $u$ stationary; (2) the fiber parameter tends to $q_0$ under domain rescaling; and (3) the weak limit retains only the principal component, whose nonzero spherical moment is no longer canceled by the exceptional bubbles.

To put things into perspective, the degeneration into a principal sphere and bubbles belongs to the
bubble-tree picture originating with Sacks--Uhlenbeck
\cite{SacksUhlenbeck1981}; see also Parker \cite{Parker1996}.  The use of a
nodal holomorphic family is also closely related to stable map and pseudoholomorphic-curve compactness; \textit{cf.} Gromov~\cite{Gromov1985} and McDuff--Salamon~\cite{McDuffSalamon2012}.
The finite-dimensional centering condition is reminiscent of Hersch's
center-of-mass normalization in \cite{Hersch1970}.

\subsection{Strategy of the proof}

Our proof of Theorem~\ref{thm:main} is constructive. 

Given a harmonic sphere $\gamma:\stwo \subset \R^3\to N$ into any Riemannian manifold $N$, its cylindrical cone
\begin{equation}
 \Phi(x,z)=\gamma\left(\frac{x}{|x|}\right),
 \qquad (x,z)\in\R^3\times\R^q
\end{equation}
is weakly harmonic. It is  stationary if and only if the following vanishes:
\begin{equation}
b(\gamma)=\int_{\stwo}\omega\,|d\gamma(\omega)|^2\,\dd\omega.
\end{equation}
This vector in $\R^3$ is the barycenter of the energy measure $|d\gamma|^2\,\dd {\rm vol}_{\Sph^2}$. The stationarity defect of $\Phi$ is supported on the spine $\{0\}\times\R^q$.

Based on the above observation, we construct a holomorphic family of spheres $\Gamma_a:\CP^1\to X$ for sufficiently small $a\in\C\setminus\{0\}$, which degenerates at $a=0$ into a nodal curve
$$M+\sum_{k=1}^4 E_k.$$
Here, $M$ is the principal component, $E_k$ are the exceptional bubbles attached to $M$ at the points lying over $p_k$, and without loss of generality we take $q_0=0$. The mappings $\Gamma_a$ arise as the blow-down of $X$, namely that $\Gamma_a (x):=\beta^{-1}(x,a)$, where $\beta$ is the blow-up of $\CP^1 \times \CP^1$
defining $X$. In addition, a K\"ahler form $\Omega_X$ is arranged so that
\begin{equation}\label{eq:smooth-fiber-balance-intro}
\int_{\stwo}\omega\,\Gamma_a^*\Omega_X=0
\end{equation}
for all sufficiently small $a\ne0$, whereas
\begin{equation}\label{eq:main-unbalanced-intro}
\int_{\stwo}\omega\,\gamma_M^*\Omega_X\ne0,
\end{equation}
where $\gamma_M$ is the inverse mapping of ${\rm pr}_1\circ \beta|_M$.

On a sufficiently small neighborhood of the origin in $\R^3_x\times\C_a\simeq\R^5$, we define away from the singular set
\begin{equation}\label{eq:u-intro}
 u(x,a)=\Gamma_a\left(\frac{x}{|x|}\right).
\end{equation}
As $\beta^{-1}$ is
jointly holomorphic in $(\omega,a)$ away from $\{(p_k, 0)\}_{k=1,\ldots,4}$, the mapping $u$ defined in \eqref{eq:u-intro} is harmonic away from its singular set. 
The balancing condition \eqref{eq:smooth-fiber-balance-intro} makes
$u$ stationary across the 3-codimensional spine $\{x=0\}$. Under rescaling, as $\rho \downarrow 0$, we have the weak $W_{\mathrm{loc}}^{1,2}$-convergence
\[
 u_\rho(x,a)=\Gamma_{\rho a}\left(\frac{x}{|x|}\right)
 \weakto
 \gamma_M\left(\frac{x}{|x|}\right).
\]
Hence, by \eqref{eq:main-unbalanced-intro}, the limit is nonstationary.

\subsection{Organization}

In \cref{sec:preliminaries}, we recall the basic variational identities for harmonic maps, derive a stationarity criterion for cylindrical cones, and state the energy-area identity for holomorphic curves in K\"ahler targets. In \cref{sec:balanced-fibration} we construct a nodal K\"ahler fibration that is crucial for our construction. In \cref{sec:map-construction}, we define the map and establish its weak harmonicity and stationarity. The tangent-map analysis and the failure of stationarity of the limit are given in \cref{sec:tangent}. We conclude with a description of the defect geometry in \cref{sec:defect-geometry}.

\section{Preliminaries}\label{sec:preliminaries}

\subsection{Weak harmonicity and stationarity}

Let $(N,h)$ be any compact smooth Riemannian manifold. By Nash's theorem, it is isometrically embedded in $\R^K$ for some $K$ as a Euclidean submanifold. Set
\begin{align*}
W^{1,2}(\Omega,N) := \Big\{v \in W^{1,2}(\Omega,\R^k):\, v(x) \in N \text{ for a.e. } x \in \Omega\Big\}.
\end{align*}
A map $u \in W^{1,2}(\Omega,N)$ is weakly harmonic if
\begin{equation}
\label{harmonic map eq} \Delta u+A_N(u)(du,du)=0
\end{equation}
in the sense of distributions, where $A_N$ is the second fundamental form of $N\subset \R^K$.  Equivalently, it is critical for compactly supported outer variations in the target variable.

A weakly harmonic map $u \in W^{1,2}(\Omega,N)$ is said to be stationary if $\cD_u(Y)=0$ for every test vector field $Y\in C_c^\infty(\Omega;\R^n)$, where \begin{equation}\label{eq:domain-first-variation}
 \cD_u(Y)=\int_\Omega \cS_{ij}(u)\partial_iY^j\,\dd x.
\end{equation}
Here, $\cS$ is (twice) the stress tensor:
\begin{equation*}
 \cS_{ij}(u)=|du|^2\delta_{ij}
 -2\langle\partial_i u,\partial_j u\rangle.
\end{equation*} 
For a smooth harmonic map, $\diver\cS(u)=0$ follows from
\begin{equation}\label{eq:stress-tension-identity}
 \partial_i\cS_{ij}(u)
 =-2\langle\tau(u),\partial_j u\rangle.
\end{equation}
For a nonsmooth weakly harmonic map, however, a distributional stress defect may remain on the singular set $$\singu:=\Big\{x \in \Omega: \text{ there is no neighborhood of $x$ on which $u$ is $C^\infty$}\Big\}.$$

\subsection{Cylindrical cones and a stationarity criterion}

Let $\sigma$ be the round metric on $\stwo$, let $\omega:\stwo\hookrightarrow\R^3$ denote the position vector. For $\gamma\in W^{1,2}(\stwo,N)$, recall the barycenter of the energy measure:
\begin{equation*}
b(\gamma)=\int_{\stwo}\omega\,|d\gamma|_\sigma^2\,\dd\vol_\sigma.
\end{equation*}

\begin{proposition}\label{prop:barycenter-criterion}
Let $\gamma:\stwo\to N$ be a smooth harmonic map.  For an integer $q\ge0$, define the cylindrical associated with $\gamma$:
\[
 \Phi(x,z)=\gamma\left(\frac{x}{|x|}\right)
 \quad\text{on }\R^3_x\times\R^q_z.
\]
Then $\Phi\in W^{1,2}_{\mathrm{loc}}$ is a weakly harmonic map, which is stationary if and only if $b(\gamma)=0$. Moreover, its inner variation is
\begin{equation}\label{eq:cone-defect}
 \cD_\Phi(Y)
 =-\int_{\R^q}b(\gamma)\cdot Y_x(0,z)\,\dd z
\end{equation}
for every compactly supported smooth vector field $Y$. 
\end{proposition}

\begin{proof}
We observe that $\Phi \in W^{1,2}_{\rm loc}$ by computing in polar coordinates $x=r\omega \in \R^3$, since $\Phi_r=0$ and $|d\Phi|^2=r^{-2}|d\gamma|_\sigma^2$. To see that $\Phi$ is weakly harmonic, note that its tension field  on $\{r>0\}$ is $\tau(\Phi)=r^{-2}\tau_{\stwo}(\gamma)=0.$ As $\Phi_r \equiv 0$, the identity~\eqref{harmonic map eq} extends across $\{x=0\}$ in the sense of distributions.

To compute the inner variation, we integrate \eqref{eq:domain-first-variation} over $\{|x|>{\eta}\}$ for $\eta>0$. One easily computes that $\cS_{rr}(\Phi)=|d\Phi|^2$ and $\cS_{r\omega}(\Phi)=\cS_{rz}(\Phi)=0$. Hence, an integration by parts yields that
\begin{align*}
\int_{\R^q}\int_{\{|x|>\eta\}} \cS_{ij}(u)\partial_iY^j\,\dd x\,\dd z = -\int_{\R^q}\int_{\stwo}
 |d\gamma(\omega)|^2
 Y_x({\eta}\omega,z)\cdot\omega
 \,\dd\omega\,\dd z.
\end{align*}
We conclude the identity~\eqref{eq:cone-defect} by sending ${\eta}\downarrow0$ and applying the dominated convergence theorem.

Finally, by choosing test vector fields $Y$ such that $Y_x(0,z)$ is parallel to
$b(\gamma)$, one deduce that $\Phi$ is stationary if and only if $b(\gamma)=0$.   \end{proof}

\subsection{Holomorphic curves in a K\"ahler target}
Denote $d^c=\frac{i}{2\pi}(\bar\partial-\partial)$ so that $dd^c=\frac{i}{\pi}\partial\bar\partial$.
Let $(X,J,h,\Omega_X)$ be K\"ahler with $\Omega_X(V,W)=h(JV,W)$. If $f:(\Sigma,j)\to(X,J)$ is holomorphic and the metric on $\Sigma$ is compatible with $j$, then $f$ is harmonic and weakly conformal.  With the usual orientation,
\begin{equation}\label{eq:energy-area}
\frac12|df|^2\,d\vol_\Sigma=f^*\Omega_X.
\end{equation}
Hence, for a holomorphic sphere $f:\stwo\to X$, it holds that
\begin{equation}\label{eq:barycenter-Kahler}
b(f)=2\int_{\stwo}\omega\,f^*\Omega_X.
\end{equation}
It follows from Proposition \ref{prop:barycenter-criterion} that the associated cylindrical
cone is stationary if and only if
\[\int_{\stwo}\omega\,f^*\Omega_X=0.
\]

\section{A nodal K\"ahler fibration}\label{sec:balanced-fibration}

In this section, we provide the key geometric construction towards the proof of Theorem~\ref{thm:main}. The main result is Proposition~\ref{prop:relative-balancing}.

\subsection{The blow-up family}  \label{subsec: the blowup family}
Let $p_1,\ldots,p_4$ be the vertices of a regular tetrahedron in $\R^3$, so that
\begin{equation}\label{eq:tetra-sum}
 \sum_{k=1}^4p_k=0.
\end{equation}
Its rotation group $G\cong A_4$  has no nonzero fixed vector in $\R^3$.

Fix $q_0\in\CP^1$ and set
\begin{equation*}
 \beta:X=
 \Bl_{(p_k,q_0),\,1\le k\le4}
 (\CP^1\times\CP^1)
 \longrightarrow\CP^1\times\CP^1,
\end{equation*}
where $\beta$ is the blow-up map. Let $\operatorname{pr}_1$ and $\operatorname{pr}_2$ denote the projections of $\mathbb{CP}^1\times\mathbb{CP}^1$ onto the first and second factors, respectively, and define
\[
q=\operatorname{pr}_1\circ\beta,
\qquad
\pi=\operatorname{pr}_2\circ\beta.
\]

For each $a\in\mathbb{CP}^1$, denote the fiber $$F_a=\pi^{-1}\{a\}.$$ For $a\ne q_0$, the restriction $q|_{F_a}:F_a\to\CP^1$ is biholomorphic, and we set
\begin{equation*}
 \Gamma_a=(q|_{F_a})^{-1}:\CP^1\longrightarrow F_a\subset X.
\end{equation*}

Let $M$ be the proper transform of
$\CP^1\times\{q_0\}$; more precisely,
\[
 M=
 \overline{
 \beta^{-1}
 \left(
 (\CP^1\times\{q_0\})
 \setminus\{(p_1,q_0),\ldots,(p_4,q_0)\}
 \right)
 }.
\]
We call $M$ the \emph{principal component} or \emph{root
component} of the central fiber. Also, let $E_k$ be the exceptional curve over
$(p_k,q_0)$.  As a divisor,
\begin{equation*}
 F_{q_0}=M+E_1+E_2+E_3+E_4.
\end{equation*}
The map $q|_M:M\to\CP^1$ is biholomorphic, while each exceptional component $E_k$ is contracted via $q$ to the singleton $\{p_k\}$. The inverse parametrization of the principal component is denoted by
\[
 \gamma_M=(q|_M)^{-1}:\CP^1\longrightarrow M.
\]
(There is no inverse of $q|_{F_{q_0}}$ on the entire central fiber,
because $q$ collapses each $E_k$ to the point $p_k$.) Observe that $M$ is the unique irreducible component of $F_{q_0}$ on which $q$ has degree one; all the $E_k$ have degree zero.

\begin{figure}[htbp]
  \centering
  \includegraphics[width=1.0\textwidth]
  {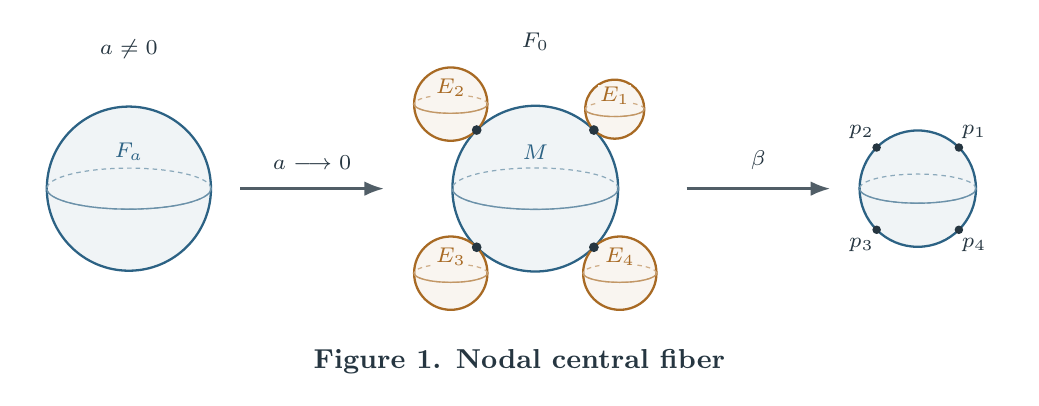}
\end{figure}

Figure~1 illustrates the degeneration, with $q_0$ represented by $a=0$ in a local coordinate on the base. The total space $X$ remains smooth.

For ease of future references, we spell out the local model.  Choose affine coordinates
$w$ and $a$ centered at $p_k$ and $q_0$, respectively.  In the blow-up
chart with coordinates $(w,t)$, the maps $\beta$, $q$, and $\pi$ are
given by
\begin{equation}\label{eq:local-blowup-chart}
 \beta(w,t)=(w,wt),\qquad q(w,t)=w,\qquad \pi(w,t)=wt.
\end{equation}
Thus $E_k=\{w=0\}$ and $M=\{t=0\}$ meet transversely at the node $n_k=\gamma_M(p_k)$, represented by $(w,t)=(0,0)$.
In particular, the critical locus of $\pi$ consists precisely of these four nodes.  
Figure~2 shows the real slices $w,t\in\mathbb R$ of the local fibers $wt=a$, in the limit $a\downarrow 0$.

\begin{figure}[htbp]
  \centering
  \includegraphics[width=0.9\textwidth]
  {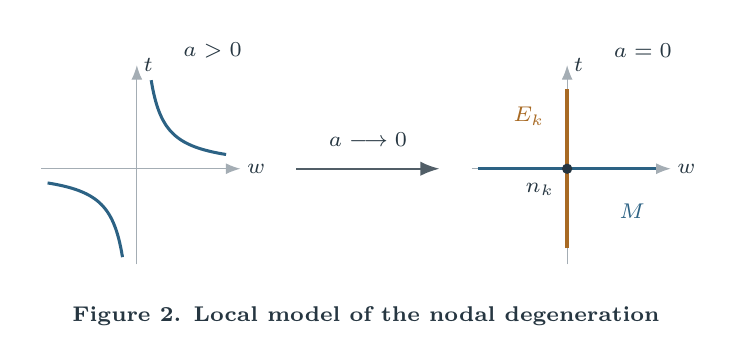}
\end{figure}

Away from the blow-up points, the inverse of $\beta$ is
\begin{equation*}
 (w,a)\longmapsto ((w,a),[w:a]).
\end{equation*}
It manifests the joint holomorphicity of $(\omega,a)\mapsto\Gamma_a(\omega)$ wherever defined.

The group $G$ acts holomorphically on the first factor of
$\CP^1\times\CP^1$, fixes the second factor, and permutes the four blow-up points.  Hence its action lifts holomorphically to $X$ and
preserves each $\pi$-fiber.

\subsection{Two auxiliary lemmas}

We first record elementary devices used to avoid nonsmooth logarithmic terms at the nodal fiber.
The first allows us to flatten a curvature representative near finitely many points.

\begin{lemma}\label{lem:flat-curvature}
Let $L\to X$ be a holomorphic line bundle over a compact complex manifold, and let $P\subset X$ be finite.  Then $c_1(L)$ has a smooth real closed $(1,1)$ representative which vanishes in a neighborhood of $P$.
\end{lemma}

\begin{proof}
Choose a Hermitian metric on $L$, and let $\Theta$ be its normalized real Chern curvature form representing $c_1(L)$. Also, choose holomorphic frames on pairwise disjoint holomorphic coordinate balls $U_p$ around the points of $P$.  On $U_p$ one can write $\Theta=\ddc\varphi_p$.  Let $\chi_p$ be supported in $U_p$ and equal to one on a smaller ball.  Then
\[
 \widetilde\Theta
 =\Theta-\ddc\left(\sum_{p\in P}\chi_p\varphi_p\right)
\]
is cohomologous to $\Theta$ and vanishes near every point of $P$.
\end{proof}

In addition, fiberwise integration of forms supported away from the nodes extend smooth across the nodes.

\begin{lemma}
\label{lem:smooth-fiber-integration}
Let $\eta$ be a smooth real $2$-form on $X$ whose support is disjoint from the critical locus of $\pi$. Then the function $I: \CP^1\setminus \{q_0\} \to \C$:
\[
I(a):=\int_{F_a}\eta
\]
extends smoothly through $a=q_0$.
\end{lemma}

\begin{proof}
Choose a coordinate disk $D$ about $q_0$ and a cutoff function $\chi\in C_c^\infty(D)$ that equals one near $q_0$. Replacing $\eta$ by $(\chi\circ\pi)\eta$ does not change  fiberwise integrals near $q_0$.

Since $X$ is compact and $\supp\eta$ is disjoint from the critical points of $\pi$, we have a finite decomposition $\eta=\sum_j\eta_j$, where each $\eta_j$ is compactly supported in a submersion chart with coordinates $(y_1,y_2,a_1,a_2)$ in which $\pi(y,a)=a$. Here $(a_1,a_2)$ are real coordinates on $D$ and $(y_1,y_2)$ are oriented fiber coordinates.
On each fiber,
\[
 \eta_j|_{F_a}
 =h_j(y,a)\,dy_1\wedge dy_2.
\]
Each $h_j$ extends by zero to a smooth compactly supported function on
$\R_y^2\times\R_a^2$. Thus, for $a\ne q_0$ near $q_0$,
\[
 \int_{F_a}\eta
 =\sum_j\int_{\R^2}h_j(y,a)\,dy.
\]
The right-hand side is well-defined at $a=q_0$, and $\partial_a^\alpha
 \sum_j\int_{\R^2}h_j(y,a)\,\dd y
 =
\sum_j\int_{\R^2}\partial_a^\alpha h_j(y,a)\,\dd y$ for every real multi-index $\alpha$. This completes the proof.   
\end{proof}

\subsection{A K\"ahler form with prescribed fiber integrals}
 
Equip $\CP^1$ with its standard complex structure and identify it with the unit round sphere by the normalized $SU(2)$-equivariant map
\[
 \mu:\CP^1\longrightarrow \stwo\subset\R^3,
\]
given in homogeneous coordinates by
\[
 \mu([z_0:z_1])
 =
 \frac{1}{|z_0|^2+|z_1|^2}
 \left(
 2\operatorname{Re}(\overline{z_0}z_1),
 2\operatorname{Im}(\overline{z_0}z_1),
 |z_0|^2-|z_1|^2
 \right).
\]
For a smooth map $f:\CP^1\to X$ and a real $(1,1)$-form $\Omega$
on $X$, define its spherical moment by
\[
 \mathbf M_\Omega(f)
 :=
 \int_{\CP^1}\mu\,f^*\Omega
 \in\R^3,
\]
where the integral is understood componentwise.  If $f$ is holomorphic
and $\Omega$ is the K\"ahler form of the target metric, then
\[
 b(f)
 :=
 \int_{\stwo}\omega\,|df|^2\,\dd \operatorname{vol}_{\stwo}
 =
 2\mathbf M_\Omega(f).
\]
Thus, $$\mathbf M_\Omega(f)=0$$ is precisely the stationarity condition
for the cylindrical cone generated by $f$.

\begin{proposition}
\label{prop:relative-balancing}
There exist a K\"ahler form $\Omega_X$ on $X$ and a coordinate disk
$\Delta$ centered at $q_0$, identified with a disk centered at
$0\in\C$, such that every noncentral smooth fiber has vanishing spherical moment:
\begin{equation}
    \label{eq:fiber-balanced}
     \mathbf M_{\Omega_X}(\Gamma_a)
 =
 \int_{\CP^1}\mu\,\Gamma_a^*\Omega_X
 =0
 \qquad
 \text{for every }a\in\Delta\setminus\{0\},
\end{equation}
whereas the principal component of the central fiber has nonvanishing spherical moment
\begin{equation}\label{eq:main-moment-nonzero}
 \mathfrak m_M:=\mathbf M_{\Omega_X}(\gamma_M)
 =
 \int_{\CP^1}\mu\,\gamma_M^*\Omega_X
 \neq0.
\end{equation}
Equivalently, the cylindrical cones generated by the smooth-fiber maps
$\Gamma_a$ are stationary, while the cone generated by
$\gamma_M$ is nonstationary.
\end{proposition}
 
\begin{proof}
Given any K\"ahler form on $X$, we average it over the finite group $G$, thus obtaining a $G$-invariant K\"ahler form $\Omega_0$. For each $a\ne0$, by construction 
\[
 \int_{\CP^1}\mu\,\Gamma_a^*\Omega_0
\]
is a fixed point by the $G$-action on $\R^3$, which must be zero.

With the integral normalization of the first Chern class,
\[c_1(\mathcal O(E_k))
   = \operatorname{PD}[E_k]
   \in H^2(X;\mathbb Z)\subset H^2(X;\mathbb R),
\]
where $\operatorname{PD}$ denotes Poincaré duality, and $E_k$, $k \in \{1,2,3,4\}$, are the exceptional spheres described in \cref{subsec: the blowup family}. Thus, for every closed oriented surface $C\subset X$, we have that
\[
\bigl\langle \operatorname{PD}[E_k],[C]\bigr\rangle
   = E_k\cdot C,
\]
where the right-hand side is the intersection number of the curves. Moreover, by virtue of \cref{lem:flat-curvature}, for each $k$ we may choose a real closed $(1,1)$ form $\Theta_k$ representing $[E_k]$, such that all the $\Theta_k$ vanish near the four nodes $M\cap E_j$ without changing their cohomology classes. It follows that
\begin{equation}\label{eq:intersection-curvature}
\int_{E_j}\Theta_k
 = \bigl\langle \operatorname{PD}[E_k],[E_j]\bigr\rangle
 = E_k\cdot E_j
 = -\delta_{kj}.
\end{equation}

Next, select real numbers $\delta_1,\dots,\delta_4$ such that
\begin{equation}\label{delta, asym}
 \sum_{k=1}^4\delta_kp_k\ne0.
\end{equation}
(Here $\delta_1=1$, $\delta_2=\delta_3=\delta_4=0$ would work.) For sufficiently small $\varepsilon>0$,
set
\begin{equation*}
 \Omega'=\Omega_0+\varepsilon\sum_{k=1}^4\delta_k\Theta_k.
\end{equation*}

For each $a\ne0$, we define
\begin{equation}\label{eq:f-a}
 f(a)=\int_{\CP^1}
 \mu\,\Gamma_a^*(\Omega'-\Omega_0)\in\R^3.
\end{equation}
Note that $\Omega'-\Omega_0$ vanishes near the critical locus of $\pi$. The same property holds for the weighted forms $(\mu_i\circ q)(\Omega'-\Omega_0)$, $i\in\{1,2,3\}$, which satisfies $$\int_{F_a}(\mu_i\circ q)(\Omega'-\Omega_0) = f_i(a)$$ by \eqref{eq:f-a}. Here and hereafter we write $\mu = [\mu_1, \mu_2, \mu_3]^\top$ and $f = [f_1,f_2,f_3]^\top$. Then, in view of \cref{lem:smooth-fiber-integration}, $f$ extends
smoothly through $a=0$.  Notice also that every $C^m$-norm of $f$ on a
fixed smaller disk is of order $O(\varepsilon)$.

We shall construct the desired K\"{a}hler form $\Omega_X$ by adding an exact correction term to $\Omega'$. To this end, recall that $\mu:\CP^1 \to \stwo$ satisfies 
\begin{equation*}
 \ddc\mu_i=-c_0\mu_i\,\dd\operatorname{vol}_\sigma
\end{equation*}
for some dimensional constant $c_0>0$, or equivalently, $-\Delta_\sigma \mu_i=2\mu_i$ for  $\sigma$ being the round metric on $\stwo$. Integration by parts gives us
\[
 \int_{\CP^1}\mu_i\,\ddc\psi
 =-c_0\int_{\CP^1}\psi\mu_i\,\dd\operatorname{vol}_\sigma
\]
for any test function $\psi$. In particular, one may pick $\psi_1, \psi_2, \psi_3 \in C^\infty(\CP^1;\R)$ supported away from the tetrahedral vertices $p_k$, such that the matrix
\begin{equation*}
\cM_{ij} := \int_{\CP^1}\mu_i\,\ddc\psi_j
\end{equation*}
is invertible.  

On a small disk $\Delta' \subset \CP^1$ about $0$, we put
\begin{equation}\label{h-tilde}
 \tilde{c}(a) := -\cM^{-1}f(a).
\end{equation}
Taking a smaller disk $\Delta$ such that $0 \in \Delta \Subset \Delta'$ and a cutoff function $\chi\equiv 1$ on $\Delta$, we may extend $\tilde{c}$ to the whole $\CP^1$: \begin{equation}\label{eq:c-a}
c = [c_1, c_2, c_3]^\top:= \chi \tilde{c} \in C^\infty\left(\CP^1;\R^3\right).
\end{equation}
Then, let us define a smooth real function on $X$, the blow-up of $\CP^1_w \times \CP^1_a$ at $\{(p_k,q_0):k=1,2,3,4\}$:
\begin{equation}\label{Psi}
\Psi(\xi):=\sum_{j=1}^3 c_j\big(\pi(\xi)\big)\,\psi_j\big(q(\xi)\big)\qquad\text{for } \xi \in X.
\end{equation}
As before, $q$ and $\pi$ are the projections onto the first and second $\CP^1$ factors of the blow-up $\beta: X \to \CP^1\times \CP^1$.

Now, we define
\begin{equation}\label{def, Omega X}
 \Omega_X:=\Omega'+\ddc\Psi.
\end{equation}
By the smallness of $f$, the boundedness of $\mathcal{M}$, and the definition of $\Omega_X$ in \eqref{def, Omega X}, we have $\|\Omega_X-\Omega_0\|_{C^\infty} \lesssim \e$. Thus, for $\e>0$ sufficiently small, $\Omega_X$ remains K\"ahler. Also observe that $\Psi$ vanishes in neighborhoods of the exceptional curves, as $\psi_j$ are chosen to be supported away from $p_k$.

When restricted to a smooth fiber $F_a$ ($a \neq 0$), all derivatives of $c$ in
the base directions disappear. It then follows that
\[
 \Gamma_a^*(\ddc\Psi)
 =\sum_{j=1}^3 c_j(a)\,\ddc\psi_j.
\]
Thanks to \eqref{eq:f-a}, \eqref{h-tilde}, and \eqref{eq:c-a}, one obtains that
\[
 \int_{\CP^1}\mu\,\Gamma_a^*\Omega_X
 =f(a)+\cM c(a)=0,
\]
which proves $\mathbf{M}_{\Omega_X}(\G_a)=0$; \textit{i.e.}, the identity~\eqref{eq:fiber-balanced}.

It remains to compute the vector $\mathfrak m_M$ associated with the main component. By
$G$-invariance, all exceptional curves have the same $\Omega_0$-area; call it $\mathcal{A}_0$. As previously commented, the function $\Psi$ defined in \eqref{Psi} vanishes near the exceptional curves. Thus, by  \eqref{eq:intersection-curvature} and \eqref{def, Omega X}, the $\Omega_X$-area of the exceptional curves can be computed as follows:
\begin{equation}\label{eq:exceptional-areas}
 \mathcal{A}_k:=\int_{E_k}\Omega_X=\mathcal{A}_0-\varepsilon\delta_k.
\end{equation}

As $a\to0$, we have the convergence of currents (denoted by $[[\cdot]]$):
\begin{equation}\label{eq:fiber-current-convergence}
[[F_a]]\longrightarrow[[M]]+\sum_{k=1}^4[[E_k]].
\end{equation}
Indeed, near each node the fiber is the divisor $wt-a=0$ in
\eqref{eq:local-blowup-chart}, where $(w,t)$ are the coordinates on the blow-up chart of $X$, in which $\beta(w,t)=(w,wt)$. The convergence $\log|wt-a|\to\log|wt|$ in
$L^1_{\mathrm{loc}}$ and the
Poincar\'{e}--Lelong formula yield the convergence to $[[w=0]]+[[t=0]]$,
each with multiplicity one.
Away from the nodes the fibers vary smoothly, so one may deduce
\eqref{eq:fiber-current-convergence} from a straightforward partition of unity argument.

Now we are ready to conclude~\eqref{eq:main-moment-nonzero}. The function $\mu\circ q$ and the form $\Omega_X$ are smooth on $X$,
and $q|_{E_k}$ is the constant function $p_k$.  Hence, passing to the limit in
\eqref{eq:fiber-balanced} leads to
\begin{equation}\label{eq:central-moment-balance}
 0=\mathfrak m_M+\sum_{k=1}^4p_k\mathcal{A}_k.
\end{equation}
But the vertices $p_k$ have center of mass zero by \eqref{eq:tetra-sum}, and the $\Omega_X$-areas of $E_k$ are determined by \eqref{eq:exceptional-areas}. Thus,
\begin{equation}\label{eq:main-moment-formula}
 \mathfrak m_M
 =\varepsilon\sum_{k=1}^4\delta_kp_k,
\end{equation}
which is nonzero due to the choice of $\delta_k$ in \eqref{delta, asym}. This completes the proof.   
\end{proof}

\begin{remark}
\label{rem:no-log}
For a generic smooth form near a nodal model $wt=a$, its fiber integral
may have terms $\sim |a|^{2j}\log|a|$. Correcting such a fiber integral by a coefficient depending only on $a$ would generally fail to give a smooth function on the total space. We circumvent this issue by choosing $\Theta_k$ vanishing near the nodes: the fiber integral error $f(a)$ is obtained by fiber integration entirely in the submersion
region and is therefore smooth.
\end{remark}

\section{Construction of the stationary map}\label{sec:map-construction}

Based on the construction of the manifold $X = {\rm Bl}_{(p_k,q_0)}(\CP^1\times \CP^1)$ and the associated K\"ahler form $\Omega_X$ in \cref{prop:relative-balancing}, in this section, we shall construct the desired stationary harmonic map $u: \cU \subset \R^5 \to X$ whose tangent map is nonstationary as in Theorem~\ref{thm:main}. Throughout this section, we let $h$ be the K\"ahler metric on $X$ associated with $\Omega_X$, and identify a coordinate neighborhood about $q_0$ with a disk $\Delta\subset\C$ centered at zero.

Let
\begin{equation*}
 \cU :={\bf B}_2(0)\times\Delta\subset\R^3_x\times\C_a\simeq\R^5.
\end{equation*}
We define $u: \cU \to X$ as follows:

\begin{itemize}
\item
For $x=r\omega\ne0$ and $(\omega,a)\ne(p_k,0)$, we set
\begin{equation}\label{eq:def-u}
 u(r\omega,a)=\Gamma_a(\omega)=\beta^{-1}(\omega,a).
\end{equation}
\item
When $a=0$ and $\omega\notin\{p_1,\dots,p_4\}$, we set $u(r\omega,0)=\gamma_M(\omega)$,  where $\gamma_M=(q|_M)^{-1}:\CP^1\rightarrow M$ and $M$ is the principal component of the central fiber of $X$, as in the previous sections. This agrees with \eqref{eq:def-u}.  
\item
Define $u$ arbitrarily on the null set where  \eqref{eq:def-u} is undefined.
\end{itemize}

In addition, put
\[
\Gamma(\omega,a)=\beta^{-1}(\omega,a)\qquad\text{on } (\stwo\times\Delta)\setminus\{(p_k,0):1\leq k\leq4\}.
\]
For $a\neq0$, this agrees with $\Gamma_a(\omega)$. Also, write $d_\omega\Gamma$ for the differential in the sphere variable and $D_a\Gamma$ for the differential in the two real coordinates of $a$.

Note that 
\begin{align*}
\singu \cap \cU \subset L \cup \bigcup_{k=1}^4 R_k,
\end{align*}
where
\begin{equation*}
 L=\{0\}\times\Delta
\end{equation*}
is the 2-dimensional spine, while
\begin{equation*}
R_k=\{(rp_k,0):0<r<2\}
\end{equation*}
are four 1-dimensional base rays.

\subsection{The inverse blow-up estimate}

We use the following standard local estimate for the inverse blow-up map in a quantitative form.

\begin{lemma}\label{lem:inverse-blowup}
Let $\widetilde{\C^2}=\Bl_0\C^2$ be equipped with any smooth Riemannian
metric.  The inverse blow-up map
\[
\mathbb{G}:\C^2\setminus\{0\}\longrightarrow\widetilde{\C^2},
 \qquad
\mathbb{G}(w,a)=((w,a),[w:a]),
\]
satisfies, for $\rho=(|w|^2+|a|^2)^{1/2}$ with
$0<\rho<\rho_0$ for some fixed $\rho_0>0$,
\begin{equation}\label{eq:inverse-blowup-pointwise}
 |d\mathbb{G}|^2\le C(1+\rho^{-2}).
\end{equation}
As a consequence,
for all sufficiently small $s>0$, it holds that
\begin{align}
&\int_{{\bf B}_s(0^4)}|d\mathbb{G}|^2\,\dd\leb^4\le Cs^2,\label{eq:inverse-blowup-energy}\\
& \int_{\partial {\bf B}_s(0^4)}|d\mathbb{G}|\,\dd\hau^3\le Cs^2,
 \qquad
 \int_{\partial {\bf B}_s(0^4)}|d\mathbb{G}|^2\,\dd\hau^3\le Cs.\label{eq:inverse-blowup-flux}
\end{align}
\end{lemma}

\begin{proof}
The only singular part is the projective coordinate $[w:a]$; let us denote $P(w,a)=[w:a]$. For a fixed normalization constant $c_{\mathrm{FS}}>0$, the Fubini--Study metric satisfies
\[
P^*g_{\mathrm{FS}} = c_{\mathrm{FS}}
\frac{|a\,dw-w\,da|^2}{(|w|^2+|a|^2)^2}.
\]
Taking the trace with respect to the Euclidean metric, one obtains
\[
|dP|^2\leq \frac{C}{|w|^2+|a|^2}.
\]
All smooth metrics are uniformly comparable on compact coordinate
neighborhoods of the exceptional divisor, so
\eqref{eq:inverse-blowup-pointwise} follows.  Integrating 
$\rho^{-2}$ against the four-dimensional volume element
$\rho^3 \dd \rho$ proves \eqref{eq:inverse-blowup-energy}. Finally, an integration on
$\partial {\bf B}_s (0^4)$ yields \eqref{eq:inverse-blowup-flux}.
\end{proof}

\subsection{Finite energy}

\begin{proposition}\label{prop:u-W12}
The map $u$ defined by \eqref{eq:def-u} belongs to
$W^{1,2}_{\mathrm{loc}}(\cU,X)$.
\end{proposition}

\begin{proof}
In polar coordinates in the $x$-variable,
\begin{equation}\label{eq:u-energy-split}
 |du|^2
 =r^{-2}|d_\omega\Gamma_a|^2+|D_a\Gamma|^2,
 \qquad u_r=0.
\end{equation}
For $a\ne0$, \eqref{eq:energy-area} leads to
\begin{equation}\label{eq:fiber-energy-constant}
 \frac12\int_{\stwo}|d_\omega\Gamma_a|^2\,\dd\omega
 =\int_{F_a}\Omega_X.
\end{equation}
The right-hand side is independent of $a$, since all smooth fibers are
homologous and $\Omega_X$ is closed.  Therefore, we have
\[
 \int_0^R\int_\Delta\int_{\stwo}
 r^{-2}|d_\omega\Gamma_a|^2r^2
 \,\dd\omega\,\dd a\,\dd r<\infty.
\]

On the other hand, away from the four points $(p_k,q_0)$, the map
$(\omega,a)\mapsto\Gamma_a(\omega)$ is smooth.  Near each such point, the local model for such this map is the inverse blow-up map of \cref{lem:inverse-blowup}; hence
\begin{equation}\label{eq:base-W12}
\int_{\Delta}\int_{\stwo}
\bigl(|D_a\Gamma|^2+|d_\omega\Gamma|^2\bigr)
\,d\omega\,da<\infty.
\end{equation}
An integration $\int_0^R$ with respect to $r^2\,\dd r$ proves the integrability of the base-derivative term in
\eqref{eq:u-energy-split}.

To identify these classical derivatives with the weak derivatives,
fix an isometric embedding $X\hookrightarrow\R^K$. The singular set $L\cup\bigcup_kR_k$ has zero $W^{1,2}$-capacity, being a finite union of smooth pieces of codimension at least three (\textit{cf. e.g.,} Evans--Gariepy \cite{eg}). One may thus choose
 smooth cutoff functions $0\le\chi_j\le1$ which vanish near this set, converge to 1 \textit{a.e.}, and 
$\|D\chi_j\|_{L^2}\to0$ locally.  Integration by parts away from the
singular set, with these cutoffs inserted, identifies the distributional derivatives of $u$ with the derivatives estimated above: the cutoff term tends to zero since $u$ is bounded, and the
remaining terms converge by the $L^2$ bounds.  This proves 
$u\in W^{1,2}_{\mathrm{loc}}(\cU,X)$.
\end{proof}

\subsection{Harmonicity away from and across the singular set}

\begin{proposition}\label{prop:u-weakly-harmonic}
The map $u$ is weakly harmonic on $\cU$.
\end{proposition}

\begin{proof}
First consider a point outside $L\cup\bigcup_kR_k$.  For fixed $r>0$,
the map
\[
 (\omega,a)\longmapsto\Gamma_a(\omega)
\]
is holomorphic from the K\"ahler product
\[
 (\CP^1\times\Delta,r^2\sigma+|da|^2)
\]
to the K\"ahler manifold $(X,h)$.  It is therefore harmonic.  Since $u_r=0$, the full tension field is
\begin{equation*}
 \tau(u)=r^{-2}\tau_\omega(\Gamma)+\tau_a(\Gamma)
 +\frac2r u_r=0.
\end{equation*}

We next remove the four base rays. On a compact subcylinder where $r\ge r_0>0$, coordinates transverse to $R_k$ consist of two angular coordinates and the two real coordinates of $a$. They are uniformly equivalent to Euclidean coordinates in $\R^4$. In light of \cref{lem:inverse-blowup}, the boundary term in the weak harmonicity identity across a transverse sphere of radius $s$ is of order $O(s^2)$, and hence tends to zero. It follows that the weak harmonic map equation~\eqref{harmonic map eq} extends across $R_k$ away from the origin.

Finally, we argue that Equation~\eqref{harmonic map eq} remains valid in the sense of distributions across $\{r=0\}$. Excise the tube $\{r<{\eta}\}$ around the spine $L$. As $\partial_r u=0$, the boundary term in the weak harmonicity identity through $\{r={\eta}\}$ vanishes. We may first remove the ray singularities for $r\ge{\eta}$ and then let ${\eta}\downarrow0$.  The finite energy established in \cref{prop:u-W12} justifies passage to the limit. This completes the proof. 
\end{proof}

\subsection{Stationarity}\label{sec:stationarity}

The crucial point is that the main component of the central fiber satisfies $b(\gamma_M)\neq0$, whereas every nearby smooth slice satisfies $b(\Gamma_a)=0$. The exceptional parameter value $a=0$ has zero 2-dimensional Lebesgue measure.

\begin{theorem}\label{thm:u-stationary}
The map $u$ is stationary on $\cU$.
\end{theorem}

\begin{proof}
On the smooth region, stationarity follows from harmonicity and
\eqref{eq:stress-tension-identity}.  

Now, fix a test vector field $$Y=(Y_x,Y_a)\in C_c^\infty(\cU;\R^3\times\R^2).$$ Choose a cylinder $Q\Subset\cU$ containing $\supp Y$ in its interior. 
Fix $\eta>0$ and apply the divergence theorem on
\[
 Q_{\eta,s}
 =\bigl(Q\cap\{r>\eta\}\bigr)
\setminus\bigcup_{k=1}^4\overline{T_s(R_k)},
\]
where $T_s(R_k)$ denotes a tubular neighborhood  about $R_k$ with transverse radius-$s$.  For $s \ll \eta$, these tubes are disjoint in $Q\cap\{r\ge\eta\}$.

By \cref{lem:inverse-blowup}, the stress tensor  of order $O_\eta(s^{-2})$ on the tube sides, whose 4-dimensional measure is $O_\eta(s^3)$.  Their total boundary contribution is therefore $O_{\eta,Y}(s)$.  The portions removed from $\{r=\eta\}$ have stress integral $O_{\eta,Y}(s^2)$, by the same 4-dimensional transverse estimate.  The boundary intersections have zero 4-dimensional measure, while the outer boundary contributes nothing because $Y$ vanishes there.

Since $\cS(u)\in L^1_{\mathrm{loc}}$, we may first let $s\downarrow0$.  This leaves only the boundary integral across $\{r=\eta\}$. We then let $\eta\downarrow0$ to conclude.

Indeed, on $\{r={\eta}\}$, direct computation yields that the inner-boundary stress-tensor integral is $-I_\eta(Y)$, where
\begin{equation}\label{eq:spine-flux-epsilon}
 I_{\eta}(Y)
 =\iint_{\Delta\times\stwo}
\left(|d_\omega\Gamma_a|^2+{\eta}^2 |D_a\Gamma|^2\right)
 Y_x({\eta}\omega,a)\cdot\omega
 \,\dd\omega\,\dd a.
\end{equation}

For $a\ne0$, the holomorphic energy identity and \eqref{eq:fiber-balanced} yield that
\begin{equation*}
 \int_{\stwo}\omega |d_\omega\Gamma_a|^2\,\dd\omega
 =2\int_{\CP^1}\mu\,\Gamma_a^*\Omega_X=0.
\end{equation*}
This holds for almost every $a\in\Delta$; the value $a=0$ is irrelevant
for integration over the parameter plane.  Therefore, \begin{align*}
 &\iint_{\Delta\times\stwo}
|d_\omega\Gamma_a|^2
 Y_x(0,a)\cdot\omega\,\dd\omega\,\dd a=0.
\end{align*}
Since $Y$ is smooth,
\[
 |Y_x({\eta}\omega,a)-Y_x(0,a)|\le C_Y{\eta}.
\]
The fiber energy in \eqref{eq:fiber-energy-constant} is uniformly
bounded, and hence the error induced by replacing
$Y_x({\eta}\omega,a)$ with $Y_x(0,a)$ in the $|d_\omega\Gamma_a|^2$ term of
\eqref{eq:spine-flux-epsilon} is of order $O({\eta})$.

For the other term, in view of~\eqref{eq:base-W12} we have
\begin{equation*}
 \left|
 {\eta}^2\iint_{\Delta\times\stwo}
 |D_a\Gamma|^2Y_x({\eta}\omega,a)\cdot\omega
 \right|\,\dd a\,\dd \omega
 \le C_Y{\eta}^2
 \iint_{\Delta\times\stwo}|D_a\Gamma|^2\,\dd  a\,\dd \omega\longrightarrow0.
\end{equation*}

Hence, $I_{\eta}(Y)\to0$, which proves that $u$ remains stationary across the spine $L$. The proof is now complete. \end{proof}

\begin{remark}\label{rem:central-slice-null}
The smooth extension of the fiber integral to $a=0$ satisfies
\[
 \mathfrak m_M+\sum_kp_k\mathcal{A}_k=0.
\]
For the map $u$ defined in \eqref{eq:def-u}, the slice $u(\cdot,0)$ contains only the main component $M$, apart from arbitrary values on the four rays, such that
\[b(\gamma_M)=2\mathfrak m_M\neq0.\] 
The values on this null slice do not affect the Sobolev class of $u$. Stationarity follows from the balancing identity on the smooth fibers and the boundary estimates in Theorem~\ref{thm:u-stationary}. After blow-up, the main component is repeated for every rescaled parameter $A\in\R^2$, and the nonzero vector $b(\gamma_M)$ becomes a genuine stationarity defect. 
\end{remark}

\section{Blow-up and the nonstationary tangent}\label{sec:tangent}

We now center the domain at $(x,a)=(0,0)$ and define
\begin{equation*}
 u_\rho(x,A)=u(\rho x,\rho A)
 =\Gamma_{\rho A}\left(\frac{x}{|x|}\right).
\end{equation*}
The map $u_\rho$ is defined on the expanding domain $\rho^{-1}\cU$. All subsequent local convergence statements are understood on fixed balls contained in this domain for sufficiently small $\rho$.

We first establish convergence of the rescaled maps.

\begin{proposition}\label{prop:tangent-convergence}
As $\rho\downarrow0$, we have $u_\rho\weakto\phi$ in $W^{1,2}_{\mathrm{loc}}(\R^5,X)$, where
\begin{equation}\label{eq:def-phi}
 \phi(x,A)=\gamma_M\left(\frac{x}{|x|}\right).
\end{equation}
\end{proposition}

It is crucial that the whole family $\{u_\rho\}$, rather than merely a subsequence,  converges to $\phi$. This ensures the uniqueness of the tangent map.

\begin{proof}
For every $A\neq 0$ and every $\omega\notin\{p_1,\dots,p_4\}$,
\begin{equation*}
 \Gamma_{\rho A}(\omega)\longrightarrow\gamma_M(\omega).
\end{equation*}
Since $X$ is compact, the dominated convergence theorem implies the strong
$L^2_{\mathrm{loc}}$ convergence. Next, thanks to
\eqref{eq:fiber-energy-constant}, 
\[
 \frac12\int_{\stwo}|d_\omega\Gamma_{\rho A}|^2\,\dd \omega
 =\int_{F_{\rho A}}\Omega_X\qquad \text{for a.e. }A.
\]
The right-hand side is the fixed K\"ahler area of the fiber class.

Moreover, for each fixed $R>0$, since 
$D_Au_\rho=\rho D_a\Gamma(\omega,\rho A)$ by the chain rule, we obtain that 
\begin{align}
\int_{|A|<R}\int_{\stwo}|D_Au_\rho|^2\,\dd\omega\,\dd A =
 \int_{|a|<\rho R}\int_{\stwo}
 |D_a\Gamma(\omega,a)|^2\,\dd\omega\,\dd a
 \longrightarrow0,
 \label{eq:a-energy-vanishes}
\end{align}
by absolute continuity of the integral in \eqref{eq:base-W12}. The radial derivatives are identically zero.  This completes the proof.
\end{proof}

\begin{theorem}\label{thm:tangent-nonstationary}
The tangent map $\phi$ is weakly harmonic but is not stationary.
\end{theorem}

\begin{proof}
First, observe that for any $\lambda>0$, we have 
\[
 \phi(\lambda x,\lambda A)=\phi(x,A);
\]
that is, $\phi$ is degree-0-homogenuous. The sphere $\gamma_M:\CP^1\to X$ is holomorphic and therefore harmonic. Hence $\phi$ is
weakly harmonic by \cref{prop:barycenter-criterion}.

Next, in view of \eqref{eq:barycenter-Kahler} and \eqref{eq:main-moment-formula},
\begin{equation*}
 b(\gamma_M)
 =2\mathfrak m_M
 =2\varepsilon\sum_{k=1}^4\delta_kp_k\ne0.
\end{equation*}
Applying \cref{prop:barycenter-criterion} with $q=2$ leads to
\begin{equation}\label{eq:phi-first-variation}
 \cD_\phi(Y)
 =-\int_{\R^2}b(\gamma_M)\cdot Y_x(0,A)\,\dd A.
\end{equation}
Choose a nonnegative nonzero $\eta\in C_c^\infty(\R^2)$ and
$\chi\in C_c^\infty(\R^3)$ with $\chi\equiv1$ near the origin. Set
\[
 Y_x(x,A)=\chi(x)\eta(A)b(\gamma_M) \quad\text{ and } \quad Y_A(x,A)\equiv 0.
\]
It then holds that the inner variation is nonvanishing:
\[
 \cD_\phi(Y)
 =-|b(\gamma_M)|^2\int_{\R^2}\eta(A)\,\dd A\ne0.
\]
Thus $\phi$ is nonstationary.   \end{proof}

Finally, we are at the stage of concluding our main theorem.

\begin{proof}[Proof of \cref{thm:main}]
By virtue of \cref{prop:u-weakly-harmonic,thm:u-stationary}, the map $u$ is a
stationary harmonic map on $\cU$.  Restrict it to a small Euclidean 5-ball in $\cU$ centered at the origin if necessary. Then, \Cref{prop:tangent-convergence,thm:tangent-nonstationary} show that its
blow-ups converge weakly to a weakly harmonic tangent map which is nonstationary.
\end{proof}

The higher-dimensional extension stated in Remark~\ref{rem:higher-dimensional-extension} follows by taking a product.
Let $n=5+\ell$ with $\ell\ge1$, choose an open ball $V\subset\mathbb R^\ell$ containing the origin, and define
\[
\widetilde{\mathcal U}=\mathcal U\times V,
\qquad
\widetilde u(y,z)=u(y).
\]
Then
\[
D_y\widetilde u=Du,\qquad D_z\widetilde u=0,
\]
so $\widetilde u$ is weakly harmonic.

Its stress tensor has the block form
\[
\mathcal S(\widetilde u)=
\begin{pmatrix}
\mathcal S(u)&0\\
0&|du|^2\operatorname{Id}_{\mathbb R^\ell}
\end{pmatrix}.
\]
Hence, for
$Y=(Y_y,Y_z)\in C_c^\infty(\widetilde{\mathcal U}; \mathbb R^5\times\mathbb R^\ell)$,
\[
\begin{aligned}
\mathcal D_{\widetilde u}(Y) ={}& \int_V
\mathcal D_u\bigl(Y_y(\cdot,z)\bigr)\,\dd z \\
&+ \int_{\mathcal U}|du(y)|^2\left(\int_V\operatorname{div}_zY_z(y,z)\,\dd z
\right)\,\dd y\\
=\ &0.
\end{aligned}
\]
Thus $\widetilde u$ is stationary.

Moreover,
\[
\widetilde u_\rho(y,z)=\widetilde u(\rho y,\rho z)=u(\rho y)\rightharpoonup\widetilde\phi(y,z):=\phi(y).
\]
Since the whole family $u_\rho$ converges to $\phi$, the tangent map
$\widetilde\phi$ is unique. Finally, if $\mathcal D_\phi(Y_0)\ne0$, choose $\chi\in C_c^\infty(\mathbb R^\ell)$ with $\int\chi=1$ and set
\[
\widetilde Y(y,z)=\bigl(\chi(z)Y_0(y),0\bigr).
\]
Then
\[
\mathcal D_{\widetilde\phi}(\widetilde Y) = \mathcal D_\phi(Y_0)\ne0.
\]
Therefore $\widetilde\phi$ is not stationary.

\section{The geometry of the defect}\label{sec:defect-geometry}

In this section, we characterise the defect measure for the blow-up of the stationary harmonic map $u$ in Theorem~\ref{thm:main}. By the seminal work~\cite{Lin1999} of Lin, the defect measure plays a crucial role in the study of geometric and analytic properties of harmonic maps. In particular, an important reason that our construction works is that $X$ carries harmonic 2-spheres.

For each $k \in \{1,2,3,4\}$, define the
three-dimensional half-space
\begin{equation*}
 H_k=\Big\{(tp_k,A) \in \R^3_x\times\R^2_A\,:t\ge0,\ A\in\R^2\Big\}.
\end{equation*}
Its boundary is the common 2-dimensional spine
\[
 L_\infty=\{0\}\times\R^2.
\]

\begin{proposition}\label{prop:scalar-defect}
The following convergence result holds locally as Radon measures on $\R^5$:
\begin{equation}\label{eq:energy-defect-measure}
 |du_\rho|^2\,dV
 \weakto
 |d\phi|^2\,dV
 +2\sum_{k=1}^4\mathcal{A}_k\,\cH^3\mres H_k.
\end{equation}
\end{proposition}

As before, $\mathcal{A}_k=\int_{E_k}\Omega_X$ is the K\"{a}hler area of the exceptional spheres $E_k$.

\begin{proof}
For any continuous function $\zeta$ on $\stwo$, by the holomorphic energy
identity and convergence of the fiber currents, we have that
\begin{align*}
 \frac12\int_{\stwo}\zeta(\omega)
 |d_\omega\Gamma_a|^2\,\dd\omega
 &=\int_{F_a}(\zeta\circ q)\,\Omega_X \\
 &\longrightarrow
 \int_M(\zeta\circ q)\,\Omega_X
 +\sum_{k=1}^4\mathcal{A}_k\zeta(p_k).                 
\end{align*}
Equivalently, as Radon measures on $\stwo$,
\begin{equation}\label{eq:angular-bubble-measure}
 |d_\omega\Gamma_a|^2\,d\omega
 \weakto
 |d\gamma_M|^2\,d\omega+2\sum_{k=1}^4\mathcal{A}_k\delta_{p_k}.
\end{equation}
Here the convergence holds for every $A\ne0$ where $a=\rho A$; the exceptional
slice $\{A=0\}$ has zero measure.  In polar coordinates in
the $x$-factor, the angular part of the energy measure is
\[
 r^{-2}|d_\omega\Gamma_{\rho A}|^2
 \bigl(r^2\,\dd r\,\dd\omega\,\dd A\bigr)
 =|d_\omega\Gamma_{\rho A}|^2\,\dd r\,\dd \omega\,\dd A.
\]
The atom at $p_k$ therefore becomes Lebesgue measure in $(r,A)$ on the
half-space $H_k$; the parametrization $(r,A)\mapsto(rp_k,A)$ is an
isometry, so this yields $\cH^3\mres H_k$.  The radial energy vanishes,
and the parameter energy tends to zero by
\eqref{eq:a-energy-vanishes}.  Thus one may conclude \eqref{eq:energy-defect-measure}.
\end{proof}

This formula gives an explicit instance of the energy identity of Naber--Valtorta \cite[Theorem~1.4]{NaberValtorta2024}. With the convention $\E(v)=\int_{\stwo}|dv|^2\,\dd\sigma$, a holomorphic parametrization of $E_k$ has energy $2\mathcal{A}_k$, which is precisely the defect density on the interior of $H_k$.

Formula
\eqref{eq:central-moment-balance} says precisely that the boundary force
of these half-spaces cancels the stationarity defect of the map part:
\begin{equation}\label{eq:defect-force-balance}
 b(\gamma_M)+2\sum_{k=1}^4\mathcal{A}_kp_k=0.
\end{equation}

To justify this cancellation at the level of the stress tensor, let
\[
 P_k=\begin{pmatrix}
 p_k\otimes p_k&0\\
 0&\Id_{\R^2}
 \end{pmatrix}
\]
be the orthogonal projection onto the tangent space of $H_k$.
Conformality of the holomorphic spheres gives, for $1\le i,j\le3$
and $x=r\omega$,
\[
 \langle\partial_{x_i}u_\rho,\partial_{x_j}u_\rho\rangle
 =\frac{|d_\omega\Gamma_{\rho A}|^2}{2r^2}
 (\delta_{ij}-\omega_i\omega_j).
\]
The terms quadratic in $D_Au_\rho$ tend to zero  in $L^1_\loc$ by
\eqref{eq:a-energy-vanishes}.  The mixed terms tend to zero by
Cauchy--Schwarz and the uniform local energy bound.  Applying
\eqref{eq:angular-bubble-measure} to the remaining angular terms, as
in the proof of \cref{prop:scalar-defect}, we infer that the
tensor-valued measure convergence
\[
 \cS(u_\rho)\,dV\weakto\cS(\phi)\,dV
 +2\sum_{k=1}^4 \mathcal{A}_kP_k\,\cH^3\mres H_k.
\]
For a compactly supported smooth vector field $Y$, define its tangential divergence along $H_k$ by
\[
\operatorname{div}_{H_k}Y:=\operatorname{tr}(P_kDY).
\]
The outward conormal of $H_k$ along $L_\infty$ is
$(-p_k,0)$, so
\[
 \int_{H_k}\diver_{H_k}Y\,\dd\cH^3
 =-\int_{\R^2}p_k\cdot Y_x(0,A)\,\dd A.
\]
Combining this identity with \eqref{eq:phi-first-variation} and
\eqref{eq:defect-force-balance} gives
\[
 \begin{aligned}
 &\cD_\phi(Y)+2\sum_{k=1}^4\mathcal{A}_k
 \int_{H_k}\diver_{H_k}Y\,\dd\cH^3
 \\
 &\qquad=-\int_{\R^2}
 \left(b(\gamma_M)+2\sum_{k=1}^4\mathcal{A}_kp_k\right)
 \cdot Y_x(0,A)\,\dd A=0.
 \end{aligned}
\]
Thus, the first variation of $\phi$ is cancelled by that of the rectifiable varifold supported on the half-spaces $H_k$ with weights $2\mathcal{A}_k$. Both first variations are nonzero, while their sum vanishes. This gives an explicit stationary map--varifold pair in the framework of Li--Zhou--Zhu~\cite{LiZhouZhu2023}.

\bigskip
\noindent
{\bf Acknowledgement}. SL is supported by NSFC Projects 12201399, 12331008, and 12411530065, Young Elite Scientists Sponsorship Program by CAST 2023QNRC001, National Key Research $\&$ Development Programs 2023YFA10-10900 and 2024YFA1014900, Shanghai Rising-Star Program 24QA2703600, Qi-Guang Scholarship, and Shanghai Frontier Research Institute for Modern Analysis.  
ZW is supported by China Postdoctoral Science Foundation 2025M773103.
	
\vspace{0.2cm}	
	
	\noindent
	{\bf Statement of competing interests}. 
	The authors declare that there is no conflict of interest.

	\noindent
	{\bf Statement of data availability}.
	Our manuscript has no associated data.

\noindent
{\bf AI Statement}. The authors thank ChatGPT 5.6 AI models for theoretical derivations and computational assistance. All mathematical derivations, conclusions, and errors remain solely the responsibility of the authors.

\end{document}